\documentclass{scrartcl}
\usepackage[numbers]{natbib}
\usepackage{fontenc}
\usepackage{shadethm}
\usepackage{amsthm}
\usepackage{float}
\usepackage{bm}
\usepackage{mathtools}
\usepackage{graphicx}
\usepackage{subfig}
\usepackage{stmaryrd}
\usepackage{mathrsfs}
\usepackage{amsmath}
\usepackage{amssymb}
\usepackage{wasysym}
\usepackage{sectsty}
\usepackage{stackengine}
\usepackage[hyperfootnotes=false]{hyperref}
\usepackage[a4paper, left=2.7cm, right=2.7cm, top=2.5cm]{geometry}
\usepackage{chngcntr}
\counterwithin*{equation}{section}

\usepackage{cleveref}
\DeclareMathAlphabet{\mathpzc}{OT1}{pzc}{m}{it}

\sectionfont{\normalsize\centering}
\subsectionfont{\footnotesize\centering}

\theoremstyle{plain}
\newtheorem{thm}{Theorem}[section] 

\theoremstyle{definition}
\newtheorem{lem}[thm]{Lemma}

\newtheorem{rem}[thm]{Remark}
\newtheorem{cor}[thm]{Corollary}

\def\XXint#1#2#3{{\setbox0=\hbox{$#1{#2#3}{\int}$ }
		\vcenter{\hbox{$#2#3$ }}\kern-.6\wd0}}

\usepackage[utf8]{inputenc}
\usepackage[german,english,russian]{babel}

\newcommand*{\Scale}[2][4]{\scalebox{#1}{$#2$}}%

\newcounter{MPequ}

\newcommand{\ssubset}{\subset\joinrel\subset}

\newcounter{AppA}

\newcounter{AppB}

\newcounter{AppC}

\newcounter{AppD}

\newcounter{AppE}

\begin{document}\selectlanguage{english}
\begin{center}
\normalsize \textbf{\textsf{Quantitative Gaffney and Korn inequalities}}
\end{center}
\begin{center}
 Wadim Gerner\footnote{\textit{E-mail address:} \href{mailto:wadim.gerner@edu.unige.it}{wadim.gerner@edu.unige.it}}
\end{center}
\begin{center}
{\footnotesize MaLGa Center, Department of Mathematics, Department of Excellence 2023-2027, University of Genoa, Via Dodecaneso 35, 16146 Genova, Italy}
\end{center}
{\small \textbf{Abstract:} 
We prove a homogeneous, quantitative version of Ehrling's inequality for the function spaces $H^1(\Omega)\ssubset L^2(\partial\Omega)$, $H^1(\Omega)\hookrightarrow L^2(\Omega)$ which reflects geometric properties of a given $C^{1,1}$-domain $\Omega\subset\mathbb{R}^n$.

We use this result to derive quantitative homogeneous versions of Gaffney's inequality, of relevance in electromagnetism as well as Korn's inequality, of relevance in elasticity theory.

The main difference to the corresponding classical results is that the constants appearing in our inequalities turn out to be dimensional constants. We provide explicit upper bounds for these constants and show that in the case of the tangential homogeneous Korn inequality our upper bound is asymptotically sharp as $n\rightarrow \infty$.

Lastly, we raise the question of the optimal values of these dimensional constants.
\section{Introduction and main results}
\subsection{Introduction}
A variation \cite[Lemma 1.5.3]{S95} of the classical Ehrling inequality \cite{Ehr54},\cite{Fich65} states that if we are given three Banach spaces $U,V,W$, a compact, linear operator $K:U\rightarrow V$ and a linear embedding $E:U\rightarrow W$, then for every $\epsilon>0$ there exists some $C_{\epsilon}>0$ such that
\begin{gather}
	\label{S1E1}
	\|K(u)\|_V\leq \epsilon \|u\|_U+C_{\epsilon}\|E(u)\|_W\text{ for all }u\in U
\end{gather}
where $\|\cdot\|_U,\|\cdot\|_V,\|\cdot\|_W$ denote the norms of $U,V,W$ respectively. This type of inequality can be used to derive a priori estimates in the context of PDEs.

To make the last statement more precise, we let $\Omega\subset\mathbb{R}^n$ be a bounded, smooth domain and consider $U=H^1(\Omega)$, $V=L^2(\partial\Omega)$ and $W=L^2(\Omega)$. It is well known that the trace map $\operatorname{tr}:H^1(\Omega)\rightarrow L^2(\partial\Omega)$ is linear and compact and obviously the inclusion map $\iota:H^1(\Omega)\rightarrow L^2(\Omega)$ is a continuous linear embedding. Applying (\ref{S1E1}) we deduce that for every $\epsilon>0$ there exists some $C_{\epsilon}>0$ such that
\begin{gather}
	\label{S1E2}
	\|u\|_{L^2(\partial\Omega)}\leq \epsilon \|u\|_{H^1(\Omega)}+C_{\epsilon}\|u\|_{L^2(\Omega)}\text{ for all }u\in H^1(\Omega).
\end{gather}
The relevance of (\ref{S1E2}) is that it can be used to control the full gradient of a vector field $B$ (under appropriate boundary conditions) by means of certain linear combinations of its partial derivatives such as $\operatorname{div}(B)\text{, }\operatorname{curl}(B)$ and $\operatorname{Sym}(\nabla B)$. Here $\operatorname{div}(B)=\sum_{i=1}^n(\partial_iB^i)$ is the standard divergence of a vector field, $(\operatorname{curl}(B))_{ij}=\frac{\partial_jB^i-\partial_iB^j}{\sqrt{2}}$, $1\leq i,j\leq n$, is the curl of the vector field and $\operatorname{Sym}(\nabla B)=\frac{DB+(DB)^{\operatorname{Tr}}}{2}$ is the symmetric gradient of the vector field, where $DB$ denotes the Jacobian of $B$ and $M^{\operatorname{Tr}}$ stands for the transpose of a matrix $M$. We remark that the factor $\sqrt{2}$ was chosen to ensure that $\sum_{i,j=1}^n\operatorname{curl}(B)^2_{ij}$ coincides with the Euclidean norm of the standard $3$-d curl vector field for $n=3$.

There are two important inequalities which follow from (\ref{S1E2}), known as Gaffney's inequality, cf. \cite{Gaff51},\cite{Fried55},\cite{MorrEElls56},\cite{Morr56}, \cite[Corollary 2.1.6 \& Theorem 2.1.7]{S95} and as Korn's inequality, cf. \cite{Korn06},\cite[\S 2]{Korn08},\cite{Korn09},\cite{Fried47}, \cite{Nit81},\cite{KondOle88},\cite{KondOle89},\cite{Ryzh99},\cite{DesVill02},\cite{BauPaul16} respectively
\begin{gather}
	\label{S1E3}
	\|B\|^2_{H^1(\Omega)}\leq C_G(\Omega,n)\left(\|\operatorname{curl}(B)\|^2_{L^2(\Omega)}+\|\operatorname{div}(B)\|^2_{L^2(\Omega)}+\|B\|^2_{L^2(\Omega)}\right)
	\\
	\nonumber
	\text{ for all }B\in (H^1(\Omega))^n\text{ with }B\parallel\partial\Omega\text{ or }B\perp \partial\Omega,
	\\
	\label{S1E4}
	\|B\|^2_{H^1(\Omega)}\leq C_K(\Omega,n)\left(\|B\|^2_{L^2(\Omega)}+\|\operatorname{Sym}(\nabla B)\|^2_{L^2(\Omega)}\right)
	\\
	\nonumber
	\text{ for all }B\in (H^1(\Omega))^n\text{ with }B\parallel\partial\Omega\text{ or }B\perp \partial\Omega
\end{gather}
for some suitable constants $C_G,C_K$ which depend on $\Omega$, the dimension of the ambient space but are independent of $B$. We emphasise that we focus in the case of Korn's inequality entirely on tangent or normal boundary conditions, see the cited references for other types of Korn inequalities.

The importance of (\ref{S1E3}) and (\ref{S1E4}) stems from the fact that in applications one frequently has only information about the curl or divergence of a vector field or the symmetric part of the gradient, while information about the full gradient is not available. In these contexts one still would like to understand the behaviour of the full gradient which can also be helpful to establish the existence of solution by means of variational techniques and regularity of solutions to certain PDEs. In the context of magnetostatics $B$ could be the magnetic field. In this case $\operatorname{div}(B)=0$ and $\operatorname{curl}(B)=J$ corresponds to the current density about which information is often available. In electrostatics $B$ may take the role of the electric field which satisfies the equations $\operatorname{curl}(B)=0$ and $\operatorname{div}(B)=\rho$ where $\rho$ is the charge density of the system. On the other hand, in the context of linear elasticity theory, $B$ may play the role of the displacement field. In this case $\operatorname{Sym}(\nabla B)$ corresponds to the (infinitesimal) strain tensor of the system. Then (\ref{S1E3}),(\ref{S1E4}) provide information about an appropriate "average" rate of change of the quantity $B$ of interest.

The constants $C_G(\Omega,n),C_K(\Omega,n)$ tell us essentially how tightly the curl, divergence and symmetric gradient can control the rate of change of the involved vector field. It is therefore of interest to understand how the constants $C_G$ and $C_K$ precisely depend on $\Omega$ and which quantities associated with $\Omega$ can control these constants. We point out that the usual proof of Ehrling's inequality argues by contradiction, see for instance \cite[Lemma 1.5.3]{S95}, so that we have no knowledge about the value of $C_{\epsilon}$ in (\ref{S1E1}) in terms of quantities associated with the operators $K$ and $E$. Consequently a prove of (\ref{S1E3}),(\ref{S1E4}) which utilises Ehrling's lemma does not provide any information about the constants $C_G$ and $C_K$.

We point out that nonetheless the question of the optimal value of $C_G$ has been investigated in \cite{Sil18}. One key observation is that if $\Omega$ is convex, then the shape operator is negative-semi definite which allows one to bypass the utilisation of Ehrling's lemma in the case where $B\parallel \partial\Omega$, cf. \cite[Theorem 2.1.5]{S95}. Similarly, if $\Omega$ is mean convex, i.e. the mean curvature of $\partial\Omega$ is non-positive, and $B\perp\partial\Omega$ one can similarly avoid utilising Ehrling's inequality because the term involving the boundary values of $B$ will be non-positive as well and can be bounded above by $0$. From this one easily obtains the fact that
\begin{gather}
	\label{S1E5}
	\|\nabla B\|^2_{L^2(\Omega)}\leq \|\operatorname{curl}(B)\|^2_{L^2(\Omega)}+\|\operatorname{div}(B)\|^2_{L^2(\Omega)}\text{ for }B\in (H^1(\Omega))^n\text{, }B\parallel \partial\Omega\text{, }\Omega\text{ is convex}
	\\
	\label{S1E6}
	\|\nabla B\|^2_{L^2(\Omega)}\leq \|\operatorname{curl}(B)\|^2_{L^2(\Omega)}+\|\operatorname{div}(B)\|^2_{L^2(\Omega)}\text{, }B\in (H^1(\Omega))^n\text{, }B\perp \partial\Omega\text{, }\Omega\text{ is mean-convex}
\end{gather}
It is shown in \cite[Theorem 8]{Sil18} that $C_G(\Omega,n)\geq 1$ and that equality holds if and only if $\Omega$ is convex (tangent to the boundary condition) and if and only if $\Omega$ is mean-convex (normal to the boundary condition) respectively. Further it is shown in \cite[Theorem 8]{Sil18} that if we let $\lambda\Omega:=\{\lambda \cdot x\mid x\in \Omega\}$ for given $\lambda>0$, then $C_G(\lambda \Omega)=C_G(\Omega)$ if and only if $C_G(\Omega)=1$. Therefore, if $\Omega$ is not mean-convex, then $C_G(\Omega)$ will depend in a non-trivial way on $\Omega$ and in particular, in stark contrast to the convex and mean-convex situation, not be an absolute constant. 

The results in \cite{Sil18} do not provide any upper bounds in terms of geometric properties of $\Omega$ for $C_G(\Omega)$ beyond the convex and mean-convex case. But by working with a partition of unity as in \cite[Theorem 3.1 \& Theorem 3.2]{Mon15}, see also \cite[Theorem 3.4]{SW21}, one can derive a quantitative version of (\ref{S1E2}) where the constant $C_{\epsilon}$ will depend also on the Lipschitz character of the domain $\Omega$ and on a chosen partition of unity. This can in turn be used to derive upper bounds on $C_G(\Omega)$ and $C_K(\Omega)$ for bounded $C^{1,1}$-domains $\Omega$ in terms of the Lipschitz character of $\partial\Omega$, the principle curvatures of $\partial\Omega$ and a choice of partition of unity. In the context of $C^{1,1}$-domains it is however more natural to work with tubular neighbourhoods which provide a "global coordinate system" around the boundary of the domain. This allows one on the one hand to avoid using a partition of unity and at the same time we are able to obtain a quantitative bound for $C_{\epsilon}$ in (\ref{S1E2}) which depends (apart from $\epsilon$) solely on the reach of $\partial\Omega$, a quantity with a clear geometric interpretation which also provides a bound on the principal curvatures of $\partial\Omega$, cf. \cite[Remark 2.8]{Dal18},\cite[Proof of Lemma 14.16]{GT01}. See also \Cref{S22} for more details about the reach of a domain and note also that the Lipschitz character alone is not enough to control principle curvatures.

This allows one to obtain bounds on the constants $C_G(\Omega,n)$ and $C_K(\Omega,n)$ in (\ref{S1E3}),(\ref{S1E4}) which depend on the reach of the domain. But these kind of estimates do not allow us to infer precise values of the optimal constants $C_G$ and $C_K$. We therefore suggest here a shift in perspective which will enable us to formulate related inequalities while at the same time removing the geometric dependence of the corresponding constant $C^*_G(\Omega,n)$, $C^*_K(\Omega,n)$ appearing in these new inequalities, i.e. we will find $C^*_G(\Omega,n)=C^*_G(n)$ and $C^*_K(\Omega,n)=C^*_K(n)$. The task then becomes to determine the best possible value of an absolute constant which appears to be more feasible.

To make this more precise, we start by letting $\operatorname{Sub}_c(\mathbb{R}^n)$ be the collection of all bounded $C^{1,1}$-domains in $\mathbb{R}^n$. We observe that if $\Omega\in \operatorname{Sub}_c(\mathbb{R}^n)$ and $\lambda>0$, we may consider the domains $\lambda\Omega$ and clearly $B(x)\mapsto B\left(\frac{x}{\lambda}\right)\equiv B_{\lambda}(x)$ maps the tangent (resp. normal) $H^1(\Omega)$-vector fields onto the tangent (resp. normal) $H^1(\lambda \Omega)$ vector fields. A simple change of variables then yields
\begin{gather}
	\nonumber
	\Scale[0.96]{\|\nabla B_{\lambda}\|^2_{L^2(\lambda \Omega)}=\lambda^{n-2}\|\nabla B\|^2_{L^2(\Omega)}\text{, }\|\operatorname{Sym}(\nabla B_{\lambda})\|^2_{L^2(\lambda\Omega)}=\lambda^{n-2}\|\operatorname{Sym}\nabla B\|^2_{L^2(\Omega)}\text{, }\|B_{\lambda}\|^2_{L^2(\lambda\Omega)}=\lambda^n\|B\|^2_{L^2(\Omega)}.}
\end{gather}
The idea now is to modify (\ref{S1E4}) in order to take into account the distinct scaling behaviour and in doing so making the constant $C_K(\Omega,n)$ scaling independent. To this end let $L:\operatorname{Sub}_c(\mathbb{R}^n)\rightarrow (0,\infty)$ be a $1$-homogeneous functional, i.e. $L(\lambda \Omega)=\lambda L(\Omega)$ for all $\lambda>0$ and all $\Omega\in \operatorname{Sub}_c(\mathbb{R}^n)$. Then replacing $\|B\|^2_{L^2(\Omega)}$ in (\ref{S1E4}) by $\frac{\|B\|^2_{L^2(\Omega)}}{L^2(\Omega)}$ will ensure that all quantities involved have the same scaling behaviour. It then follows from (\ref{S1E4}) that for any such $1$-homogeneous function $L$ we have an inequality of the form
\begin{gather}
	\label{S1E7}
	\frac{\|B\|^2_{L^2(\Omega)}}{L^2(\Omega)}+\|\nabla B\|^2_{L^2(\Omega)}\leq C^L_K(\Omega,n)\left(\frac{\|B\|^2_{L^2(\Omega)}}{L^2(\Omega)}+\|\operatorname{Sym}(\nabla B)\|^2_{L^2(\Omega)}\right)\text{ provided }B\parallel \partial\Omega\text{ or }B\perp \partial\Omega
\end{gather}
where the constant $C^L_K(\Omega,n)$ will depend on $L$. If we let $C^L_K(\Omega,n)$ be the best possible constant in the inequality (\ref{S1E7}) then by our choice we clearly have $C^L_K(\lambda\Omega,n)=C^L_K(\Omega,n)$ for all $\lambda>0$ and all $\Omega\in \operatorname{Sub}_c(\mathbb{R}^n)$. This does not show that $C^L_K(\Omega,n)$ will be independent of $\Omega$ as it could still depend on quantities like $\frac{|\Omega|^{\frac{1}{n}}}{\operatorname{diam}(\Omega)}$ or similar. The hope however is that for the right choice of $L$ it may turn out that $C^L_K(\Omega,n)$ becomes indeed independent of $\Omega$ which would allow us to collect the geometric dependence of $C^L_K(\Omega,n)$ in the prefactor $L(\Omega)$ while turning $C^L_K(\Omega,n)$ into a dimension dependent constant.

Identifying such an $L$ is in itself a non-trivial task and the right choice of $L$ will depend on the specific inequality under consideration. As we shall see, if we let $\rho(\partial\Omega)\equiv \rho$ denote the reach of the boundary of $\Omega$ (which is a well-defined quantity whenever $\partial\Omega\in C^{1,1}$) then we are able to turn $C^*_K(\Omega,n):=C^{\rho}_K(\Omega,n)$ into an absolute dimensional constant. The same turns out to be true for (\ref{S1E3}). We hence will be able to show the following
\begin{gather}
	\label{S1E8}
	\frac{\|B\|^2_{L^2(\Omega)}}{\rho^2(\partial\Omega)}+\|\nabla B\|^2_{L^2(\Omega)}\leq C^*_G(n)\left(\frac{\|B\|^2_{L^2(\Omega)}}{\rho^2(\partial\Omega)}+\|\operatorname{curl}(B)\|^2_{L^2(\Omega)}+\|\operatorname{div}(B)\|^2_{L^2(\Omega)}\right)
	\\
	\nonumber
	\text{for all }\Omega\in \operatorname{Sub}_c(\mathbb{R}^n)\text{ and all }B\in \left(H^1(\Omega)\right)^n\text{ with }B\parallel \partial\Omega\text{ or }B\perp \partial\Omega,
	\\
	\label{S1E9}
	\frac{\|B\|^2_{L^2(\Omega)}}{\rho^2(\partial\Omega)}+\|\nabla B\|^2_{L^2(\Omega)}\leq C^*_K(n)\left(\frac{\|B\|^2_{L^2(\Omega)}}{\rho^2(\partial\Omega)}+\|\operatorname{Sym}(\nabla B)\|^2_{L^2(\Omega)}\right)
\\
\nonumber
\text{for all }\Omega\in \operatorname{Sub}_c(\mathbb{R}^n)\text{ and all }B\in \left(H^1(\Omega)\right)^n\text{ with }B\parallel \partial\Omega\text{ or }B\perp \partial\Omega.
\end{gather}
Our approach is constructive in the sense that we are able to provide explicit upper bounds on the constants $C^*_G(n)$ and $C^*_K(n)$. Moreover we point out that the optimal value of the constants $C^*_G(n)$ and $C^*_K(n)$ may depend on the prescribed boundary conditions (tangent or normal). We finally observe that one can easily obtain the following relationship between the optimal constants $C_G(\Omega,n)$,$C_K(\Omega,n)$ and $C^*_G(n)$,$C^*_K(n)$
\begin{gather}
	\label{S1E10}
	C_G(\Omega,n)\leq \max\left\{\rho^2(\partial\Omega),\rho^{-2}(\partial\Omega)\right\}C^*_G(n)
	\\
	\label{S1E11}
	C_K(\Omega,n)\leq \max\left\{\rho^2(\partial\Omega),\rho^{-2}(\partial\Omega)\right\}C^*_K(n).
\end{gather}
These inequalities in particular imply that $C_G(\Omega,n)\leq C^*_G(n)$ if $\Omega$ is of unit reach, i.e. $\rho(\partial\Omega)=1$. Similarly $C_K(\Omega,n)\leq C^*_K(n)$ if $\Omega$ is of unit reach. Therefore the constants $C^*_G(n)$ and $C^*_K(n)$ provide upper bounds on the optimal constants in (\ref{S1E3}) and (\ref{S1E4}) among domains of unit reach respectively.
\section{Preliminaries}
\subsection{Notation}
\label{S21}
Throughout $n$ will always denote a positive integer which is greater or equal to $2$. Given a bounded $C^{1,1}$-domain $\Omega\subset\mathbb{R}^n$ we denote by $\kappa_i$, $1\leq i\leq n-1$, $s$ and $H$ the principal curvatures, the shape operator and the mean curvature of $\partial\Omega$ with respect to the outward unit normal and we note that $\kappa_i\in L^{\infty}(\partial\Omega)$. Here we use the convention $H=\frac{\sum_{i=1}^{n-1}\kappa_i}{n-1}$. We denote by $\rho(\partial\Omega)$ the reach of $\partial\Omega$, see also \Cref{S22} for more details about the reach of the boundary of a domain. We denote by $H^1(\Omega)$ the space of scalar $H^1$-functions on $\Omega$ and by $\left(H^1(\Omega)\right)^n$ the space of $H^1$-vector fields on $\Omega$. Due to the trace inequality each $B\in \left(H^1(\Omega)\right)^n$ has a well defined trace $B|_{\partial\Omega}\in \left(L^2(\partial\Omega)\right)^n$. We say that $B\in \left(H^1(\Omega)\right)^n$ is parallel to $\partial\Omega$ and write $B\parallel \partial\Omega$ if $\mathcal{N}(x)\cdot B|_{\partial\Omega}(x)=0$ for $\mathcal{H}^{n-1}$-a.e. $x\in \partial\Omega$ where $\mathcal{N}$ denotes the outward pointing unit normal on $\mathcal{N}$ and $\mathcal{H}^{n-1}$ denotes the $(n-1)$-dimensional Hausdorff measure. We say that $B\in \left(H^1(\Omega)\right)^n$ is normal to $\partial\Omega$ and write $B\perp \partial\Omega$ if $B|_{\partial\Omega}=(B|_{\partial\Omega}\cdot \mathcal{N})\mathcal{N}$. We let $e_i$, $1\leq i\leq n$ denote the standard basis vectors of $\mathbb{R}^n$. We decompose $B\in \left(H^1(\Omega)\right)^n$ as $B=\sum_{i=1}^nB^ie_i$ and define $\|B\|^2_{L^2(\Omega)}:=\sum_{i=1}^n\|B^i\|^2_{L^2(\Omega)}$, $\|\nabla B\|^2_{L^2(\Omega)}:=\sum_{i,j=1}^n\|\partial_jB^i\|^2_{L^2(\Omega)}$, $\|B\|^2_{H^1(\Omega)}:=\|B\|^2_{L^2(\Omega)}+\|\nabla B\|^2_{L^2(\Omega)}$. If $A\in \left(H^1(\Omega)\right)^{n\times n}$ is a square matrix with $H^1(\Omega)$ entries $A_{ij}$, then we set $\|A\|^2_{L^2(\Omega)}:=\sum_{i,j=1}^n\|A_{ij}\|^2_{L^2(\Omega)}$. Lastly, we define for given $B\in \left(H^1(\Omega)\right)^n$, $(\nabla B)_{ij}:=(\partial_jB^i)$, $(\operatorname{Sym}(\nabla B))_{ij}:=\frac{(\partial_jB^i)+(\partial_iB^j)}{2}$ and $(\operatorname{curl}(B))_{ij}:=\frac{(\partial_jB^i)-(\partial_iB^j)}{\sqrt{2}}$. We note that with our definition $\operatorname{curl}(B)$ does not coincide with the anti-symmetric part of $\nabla B$. The reason for this choice is that with our definition the $L^2$-norm of $\operatorname{curl}(B)$ for $n=3$ coincides with the $L^2$-norm of the standard vector curl of $B$ which appears in the standard formulation of Maxwell's equations.
\subsection{The reach of a domain}
\label{S22}
The notion of the reach of a set dates back to the work of Federer \cite[Section 4]{Fed59}. We consider here only the situation where $\Omega \subset \mathbb{R}^n$ is a bounded $C^{1,1}$-domain. We define the distance function $\delta_{\partial\Omega}(x):=\operatorname{dist}(x,\partial\Omega)=\inf_{y\in \partial\Omega}|x-y|$. The set of points which admit a unique projection onto $\partial\Omega$ is defined by $\operatorname{Unp}(\partial\Omega):=\{x\in \mathbb{R}^n\mid \exists! y\in \partial\Omega:\delta_{\partial\Omega}(x)=|x-y|\}$. The reach of $\partial\Omega$ at a point $y\in \partial\Omega$ is defined by $\operatorname{reach}(\partial\Omega,y):=\sup\{0<r\mid B_r(y)\subset \operatorname{Unp}(\partial\Omega)\}$ and finally the reach of $\partial\Omega$ is defined by $\rho(\partial\Omega)\equiv \operatorname{reach}(\partial\Omega):=\inf_{y\in \partial\Omega}\operatorname{reach}(\partial\Omega,y)$. Geometrically, the reach of $\partial\Omega$ is the largest number such that every ball of a radius smaller than $\operatorname{reach}(\partial\Omega)$ centred around any point $y\in \partial\Omega$ admits a unique projection onto $\partial\Omega$. We denote the corresponding projection map as follows $P_{\partial\Omega}:\operatorname{Unp}(\partial\Omega)\rightarrow\partial\Omega$, $x\mapsto y$ where $y\in \partial\Omega$ is the unique point which realises the distance to $\partial\Omega$.
\newline
\newline
A strongly related notion is the notion of a uniform ball domain. We say that $\Omega$ satisfies the $\epsilon$-uniform ball condition for some $0<\epsilon$ if for every $y\in \partial\Omega$, there exists some $d_y\in S^{n-1}$ with $B_{\epsilon}(y-\epsilon d_y)\subset\Omega$ and $B_{\epsilon}(y+\epsilon d_y)\subset \mathbb{R}^n\setminus \overline{\Omega}$ where $S^m$ denotes the $m$-dimensional unit sphere and $B_r(x)$ denotes the open ball of radius $r$ centred at $x$. We have the following relationship between the notion of the reach and the notion of a uniform ball domain
\begin{lem}[{\cite[Theorem 2.5 \& Theorem 2.6]{Dal18}}]
	\label{S2L1}
	Let $\Omega\subset\mathbb{R}^n$ be a bounded $C^{1,1}$-domain, then the following holds
	\begin{enumerate}
		\item $\rho(\partial\Omega)=\sup\{\epsilon>0\mid \Omega\text{ satisfies the }\epsilon\text{-uniform ball condition}\}$,
		\item $0<\operatorname{\rho}(\partial\Omega)<\infty$,
		\item $d_y$ in the definition of the $\epsilon$-ball condition can be taken to be $d_y=\mathcal{N}(y)$, where $\mathcal{N}(y)$ denotes the outward unit normal.
	\end{enumerate}
\end{lem}
We point out that another characterisation of the reach in terms of properties of the signed distance function can also be obtained, cf. \cite[Theorem 2.7]{Dal18}. For our purposes the precise relationship is not of relevance, but we will make use of the signed distance function which is defined as follows
\begin{gather}
	\nonumber
	b_{\Omega}(z):=\begin{cases}
		\operatorname{dist}(z,\partial\Omega) & \text{ if }z\in \mathbb{R}^n\setminus \overline{\Omega} \\
		0 &\text{ if }z\in \partial\Omega \\
		-\operatorname{dist}(z,\partial\Omega) &\text{ if }z\in \Omega
	\end{cases}
\end{gather}
For our purposes we will need the following properties of the reach whose proof we provide for completeness
\begin{lem}[Reach and tubular neighbourhoods]
	\label{S2L2}
	Let $\Omega\subset\mathbb{R}^n$ be a bounded $C^{1,1}$-domain. Define $I:=(-\rho(\partial\Omega),\rho(\partial\Omega))$ and the map $\Psi:I\times \partial\Omega\rightarrow \mathbb{R}^n\text{, }(t,y)\mapsto y+t\mathcal{N}(y)$. Then the following holds
	\begin{enumerate}
		\item $\Psi$ defines a bi-Lipschitz homeomorphism onto its image,
		\item $\Psi\left(I\times \partial\Omega\right)=\{x\in \mathbb{R}^n\mid \delta_{\partial\Omega}(x)<\rho(\partial\Omega)\}$,
		\item If $\partial\Omega\in C^k$ for some $k\geq 2$, then $\Psi$ is a $C^{k-1}$-diffeomorphism.
	\end{enumerate}
\end{lem}
\begin{proof}[Proof of \Cref{S2L2}]
	We fix $0<\epsilon<\rho(\partial\Omega)$, set $I_{\epsilon}:=(-\epsilon,\epsilon)$ and consider $\Psi_{\epsilon}:=\Psi|_{I_{\epsilon\times \partial\Omega}}$. It is then enough to establish the lemma for any such fixed $\epsilon$. Clearly $\Psi_{\epsilon}$ is Lipschitz continuous and of class $C^{k-1}$ if $\partial\Omega\in C^k$. Since $\epsilon<\rho(\partial\Omega)$ it follows from (i) and (ii) of \Cref{S2L1} that $|y+t\mathcal{N}(y)-z|\geq |t|$ for every $|t|\leq \epsilon$ and all $z\in \partial\Omega$. Further, $|y+t\mathcal{N}(y)-y|=|t|$ so that $y$ realises the distance between $y+t\mathcal{N}(y)$ and $\partial\Omega$. Using once more that $\epsilon<\rho(\partial\Omega)$ we deduce $y+t\mathcal{N}(y)\in \operatorname{Unp}(\partial\Omega)$ and thus $P_{\partial\Omega}(y+t\mathcal{N}(y))=y$ for every $y\in \partial\Omega$ and all $|t|<\epsilon$, where we recall that $P_{\partial\Omega}$ denotes the nearest point projection. To prove injectivity of $\Psi_{\epsilon}$, we observe that if $y+t\mathcal{N}(y)=z+s\mathcal{N}(z)$ for $|t|,|s|<\epsilon$, $y,z\in \partial\Omega$, then $y=P_{\partial\Omega}(y+t\mathcal{N}(y))=P_{\partial\Omega}(z+s\mathcal{N}(z))=z$. Knowing that $y=z$ we find $y+t\mathcal{N}(y)=z+s\mathcal{N}(z)=y+s\mathcal{N}(y)$ which immediately yields $t=s$ and establishes injectivity of $\Psi_{\epsilon}$. To characterise the image we notice that we had already established that $\delta_{\partial\Omega}(y+t\mathcal{N}(y))=|t|$ for every $|t|<\epsilon$ and thus $\Psi_{\epsilon}(I_{\epsilon}\times \partial\Omega)\subset \{x\in \mathbb{R}^n\mid \delta_{\partial\Omega}(x)<\epsilon\}$. Conversely, if $w\in \{x\in \mathbb{R}^n\mid \delta_{\partial\Omega}(x)<\epsilon\}$, then since $\epsilon<\rho(\partial\Omega)$ we deduce that $w\in \operatorname{Unp}(\partial\Omega)$ and thus admits a unique projection $z=P_{\partial\Omega}(w)\in \partial\Omega$. Consequently $w\in B_{\epsilon}(z)$ and it then follows from \cite[Equation (A.6)]{Dal18} that $w=z+b_{\Omega}(w)\mathcal{N}(z)$. Since $|b_{\Omega}(w)|=\delta_{\partial\Omega}(w)<\epsilon$ we conclude that $w\in \Psi_{\epsilon}\left(I_{\epsilon}\times \partial\Omega\right)$ which overall shows that $\Psi_{\epsilon}\left(I_{\epsilon}\times \partial\Omega\right)=\{x\in \mathbb{R}^n\mid \delta_{\partial\Omega}(x)<\epsilon\}$. We further notice that the inverse of $\Psi_{\epsilon}$ is given by $(b_{\Omega}\times P_{\partial\Omega})(w)=(b_{\Omega}(w),P_{\partial\Omega}(w))$ and that $b_{\Omega}$ and $P_{\partial\Omega}$ are Lipschitz continuous, cf. \cite[Theorem 5.1]{DelZol94} and \cite[Equation (A.7)]{Dal18} so that we overall conclude that if $0<\epsilon<\rho(\partial\Omega)$, then $\Psi_{\epsilon}$ defines a bi-Lipschitz homeomorphism onto $\{x\in \mathbb{R}^n\mid \delta_{\partial\Omega}(x)<\epsilon\}$.
	We are left with observing that if $\partial\Omega\in C^k$, then $b_{\Omega}\in C^k$ and $P_{\partial\Omega}\in C^{k-1}$, cf. \cite[Theorem 5.5, Theorem 5.6]{DelZol94}, \cite[Lemma 14.16]{GT01}, which proves the lemma.
\end{proof}
\section{Main results}
\subsection{A homogeneous Gaffney inequality}
\label{S31}
\begin{thm}[Homogeneous Gaffney inequality]
	\label{S3T1}
	Let $\Omega\subset\mathbb{R}^n$ be a bounded domain with $C^{1,1}$-boundary and let $\rho(\partial\Omega)$ denote the reach of $\partial\Omega$. Then for all $B\in \left(H^1(\Omega)\right)^n$ the following holds
	\begin{gather}
		\label{S3E1}
		\frac{\|B\|^2_{L^2(\Omega)}}{\rho^2(\partial\Omega)}+\|\nabla B\|^2_{L^2(\Omega)}\leq C_1(n)\left(\frac{\|B\|^2_{L^2(\Omega)}}{\rho^2(\partial\Omega)}+\|\operatorname{curl}(B)\|^2_{L^2(\Omega)}+\|\operatorname{div}(B)\|^2_{L^2(\Omega)}\right)\text{ if }B\parallel \partial\Omega
		\\
		\label{S3E2}
		\frac{\|B\|^2_{L^2(\Omega)}}{\rho^2(\partial\Omega)}+\|\nabla B\|^2_{L^2(\Omega)}\leq C_2(n)\left(\frac{\|B\|^2_{L^2(\Omega)}}{\rho^2(\partial\Omega)}+\|\operatorname{curl}(B)\|^2_{L^2(\Omega)}+\|\operatorname{div}(B)\|^2_{L^2(\Omega)}\right)\text{ if }B\perp \partial\Omega
	\end{gather}
	where $C_1(n):=1+(1+\sqrt{1+n})^2$ and $C_2(n):=1+(n-1+\sqrt{n-1}\sqrt{2n-1})^2$.
\end{thm}
In the following let $C^T_G(\Omega,n)$ and $C^N_G(\Omega,n)$ denote the constant $C_G(\Omega,n)$ from (\ref{S1E3}) under the boundary conditions $B\parallel \partial\Omega$ and $B\perp \partial\Omega$ respectively. It then follows from (\ref{S1E10})
\begin{cor}[Estimates for Gaffney's constant]
	\label{S3C2}
	Let $\Omega\subset\mathbb{R}^n$ be a bounded domain with $C^{1,1}$-boundary and let $\rho(\partial\Omega)$ denote the reach of $\partial\Omega$. Then we have the following inequalities
	\begin{gather}
		\label{S3E3}
		C^T_G(\Omega,n)\leq C_1(n)\max\{\rho^2(\partial\Omega),\rho^{-2}(\partial\Omega)\}\text{, }C^N_G(\Omega,n)\leq C_2(n)\max\{\rho^2(\partial\Omega),\rho^{-2}(\partial\Omega)\}.
	\end{gather}
\end{cor}
\begin{rem}
	\label{S3R3}
\begin{enumerate}
	\item We use the notation $C^{T,*}_G(n)$,$C^{N,*}_G(n)$ for the optimal values of the constants in (\ref{S1E8}) under tangent and normal boundary conditions respectively. If $\Omega\subset\mathbb{R}^n$ is any bounded $C^{1,1}$-domain and $B\in \left(H^1(\Omega)\right)^n$ is any $H^1$-vector field with $B|_{\partial\Omega}=0$, then one readily checks that $\|\nabla B\|^2_{L^2(\Omega)}=\|\operatorname{curl}(B)\|^2_{L^2(\Omega)}+\|\operatorname{div}(B)\|^2_{L^2(\Omega)}$ which implies $C^{T,*}_G(n)\geq 1$ and $C^{N,*}_G(n)\geq 1$ for all $n$.
	\item More restrictive lower bounds can be obtained by introducing the spaces of harmonic Neumann fields $\mathcal{H}_N(\Omega)$ and harmonic Dirichlet fields $\mathcal{H}_D(\Omega)$ respectively as follows
	\begin{gather}
		\nonumber
		\mathcal{H}_N(\Omega):=\left\{\Gamma\in \left(H^1(\Omega)\right)^n\mid \operatorname{curl}(\Gamma)=0\text{, }\operatorname{div}(\Gamma)=0\text{, }\Gamma\parallel\partial\Omega\right\},
		\\
		\nonumber
		\mathcal{H}_D(\Omega):=\left\{\Gamma\in \left(H^1(\Omega)\right)^n\mid \operatorname{curl}(\Gamma)=0\text{, }\operatorname{div}(\Gamma)=0\text{, }\Gamma\perp\partial\Omega\right\}.
	\end{gather} 
	It is well-known that $\dim(\mathcal{H}_N(\Omega))=\operatorname{dim}\left(H^1_{\operatorname{dR}(\Omega)}\right)$ and $\dim(\mathcal{H}_D(\Omega))=\dim\left(H^{n-1}_{\operatorname{dR}}(\Omega)\right)$, cf. \cite[Theorem 2.6.1 \& Corollary 2.6.2]{S95}, where $H^k_{\operatorname{dR}}(\Omega)$ denotes the $k$-th de Rham cohomology group of $\Omega$. It then immediately follows from (\ref{S1E8}) that we have the following lower bounds
	\begin{gather}
		\label{S3Extra1}
		1+\sup_{\Omega\in \operatorname{Sub}_c(\mathbb{R}^n)}\sup_{\Gamma\in \mathcal{H}_N(\Omega)\setminus \{0\}}\frac{\rho^2(\partial\Omega)\|\nabla \Gamma\|^2_{L^2(\Omega)}}{\|\Gamma\|^2_{L^2(\Omega)}}\leq C^{T,*}_G(n),
		\\
		\label{S3Extra2}
		1+\sup_{\Omega\in \operatorname{Sub}_c(\mathbb{R}^n)}\sup_{\Gamma\in \mathcal{H}_D(\Omega)\setminus \{0\}}\frac{\rho^2(\partial\Omega)\|\nabla \Gamma\|^2_{L^2(\Omega)}}{\|\Gamma\|^2_{L^2(\Omega)}}\leq C^{N,*}_G(n).
	\end{gather}
	As an application of these formulas we can consider the physically most relevant case $n=3$. If $\Omega\subset\mathbb{R}^3$ is rotationally symmetric about the $z$-axis and of positive distance to the $z$-axis, then we can define the vector fields $Y(x,y,z):=(-y,x,0)$ and $\Gamma:=\frac{Y}{|Y|^2}$. Then, since $Y$ induces rotations around the $z$-axis we find $Y\parallel \partial\Omega$ and thus $\Gamma\parallel \partial\Omega$. Further, a direct computation yields $\operatorname{div}(\Gamma)=0=\operatorname{curl}(\Gamma)$ and so $\Gamma\in \mathcal{H}_N(\Omega)$. To obtain an explicit estimate we may take $\Omega$ to be a rotationally symmetric solid torus with minor radius $r$ and major radius $R$, $0<r<R<\infty$, and of aspect ratio $a:=\frac{R}{r}$. The reach of $\partial\Omega$ is then simply given by $\rho(\partial\Omega)=\min\{r,R-r\}=r\min\{1,a-1\}$ and the quantities $\|\Gamma\|^2_{L^2(\Omega)}$ and $\|\nabla \Gamma\|^2_{L^2(\Omega)}$ admit a closed form expression in terms of the aspect ratio and the radii $r,R$. One can verify that the maximum of the expression $\frac{\rho^2(\partial\Omega)\|\nabla \Gamma\|^2_{L^2(\Omega)}}{\|\Gamma\|^2_{L^2(\Omega)}}$ is achieved at an aspect ratio of $2$ and obtains from this the following non-trivial lower bound
	\begin{gather}
		\label{S3Extra3}
		C^{T,*}_G(3)\geq 1+\frac{1-\frac{2}{\sqrt{3}^3}}{6(2-\sqrt{3})}\approx 1.38259\dots
	\end{gather}
	By considering other rotationally symmetric domains one may potentially increase the lower bound on $C^{T,*}_G(3)$.
	\item If $n=2$, $\Omega\subset\mathbb{R}^2$ is a bounded $C^{1,1}$-domain and $B=(B^1,B^2)\in \left(H^1(\Omega)\right)^2$ is tangent to the boundary, then $B^\perp:=(-B^2,B^1)$ is normal to the boundary with $\|B^\perp\|_{L^2(\Omega)}=\|B\|_{L^2(\Omega)}$, $\|\nabla B^\perp\|_{L^2(\Omega)}=\|\nabla B\|_{L^2(\Omega)}$,$\|\operatorname{curl}(B^\perp)\|_{L^2(\Omega)}=\|\operatorname{div}(B)\|_{L^2(\Omega)}$,$\|\operatorname{div}(B^\perp)\|_{L^2(\Omega)}=\|\operatorname{curl}(B)\|_{L^2(\Omega)}$ from which one easily deduces that $C^{T,*}_G(2)=C^{N,*}_G(2)$.
\end{enumerate}
\end{rem}
\textbf{Open Question 1:} What is the exact value of the optimal constant $C^{T,*}_G(n)$ in the homogeneous Gaffney inequality (\ref{S1E8}) under tangent boundary conditions?
\newline
\newline
\textbf{Open Question 2:} What is the exact value of the optimal constant $C^{N,*}_G(n)$ in the homogeneous Gaffney inequality (\ref{S1E8}) under normal boundary conditions?
\subsection{A homogeneous Korn inequality}
\label{S32}
\begin{thm}[Homogeneous Korn inequality]
	\label{S3T4}
	Let $\Omega\subset\mathbb{R}^n$ be a bounded domain with $C^{1,1}$-boundary and let $\rho(\partial\Omega)$ denote the reach of $\partial\Omega$. Then for all $B\in \left(H^1(\Omega)\right)^n$ the following holds
	\begin{gather}
		\label{S3E4}
		\frac{\|B\|^2_{L^2(\Omega)}}{\rho^2(\partial\Omega)}+\|\nabla B\|^2_{L^2(\Omega)}\leq C_1(n)\left(\frac{\|B\|^2_{L^2(\Omega)}}{\rho^2(\partial\Omega)}+\|\operatorname{Sym}(\nabla B)\|^2_{L^2(\Omega)}\right) \text{ if }\parallel \partial\Omega,
		\\
		\label{S3E5}
		\frac{\|B\|^2_{L^2(\Omega)}}{\rho^2(\partial\Omega)}+\|\nabla B\|^2_{L^2(\Omega)}\leq C_2(n)\left(\frac{\|B\|^2_{L^2(\Omega)}}{\rho^2(\partial\Omega)}+\|\operatorname{Sym}(\nabla B)\|^2_{L^2(\Omega)}\right)\text{ if }B\perp\partial\Omega
	\end{gather}
	with $C_1(n)$,$C_2(n)$ as in \Cref{S3T1}.
\end{thm}
In the following let $C^T_K(\Omega,n)$ and $C^N_K(\Omega,n)$ denote the optimal constant $C_K(\Omega,n)$ from (\ref{S1E4}) under the boundary conditions $B\parallel \partial\Omega$ and $B\perp\partial\Omega$ respectively. It then follows from (\ref{S1E11})
\begin{cor}[Estimates for Korn's constant]
	\label{S3C5}
	Let $\Omega\subset\mathbb{R}^n$ be a bounded domain with $C^{1,1}$-boundary and let $\rho(\partial\Omega)$ denote the reach of $\partial\Omega$. Then we have the following inequalities
	\begin{gather}
		\label{S3E6}
		C^T_K(\Omega,n)\leq C_1(n)\max\{\rho^2(\partial\Omega),\rho^{-2}(\partial\Omega)\}\text{, }C^N_K(\Omega,n)\leq C_2(n)\max\{\rho^2(\partial\Omega),\rho^{-2}(\partial\Omega)\}.
	\end{gather}
\end{cor}
\begin{rem}
	\label{S3R6}
	\begin{enumerate}
		\item We use the notation $C^{T,*}_K(n)$,$C^{N,*}_K(n)$ for the optimal constants in (\ref{S1E9}) under tangent and normal boundary conditions respectively. If $\Omega\subset\mathbb{R}^n$ is any bounded $C^{1,1}$-domain and $B\in \left(H^1(\Omega)\right)^n$ is any vector field with $B|_{\partial\Omega}=0$, then one easily checks that $\|\nabla B\|^2_{L^2(\Omega)}=2\|\operatorname{Sym}\nabla B\|^2_{L^2(\Omega)}-\|\operatorname{div}(B)\|^2_{L^2(\Omega)}$. So, if $B=\operatorname{curl}(A)$ for some $(H^2(\Omega))^n$ vector field which is compactly supported in $\Omega$ we find $\|\nabla B\|^2_{L^2(\Omega)}=2\|\operatorname{Sym}\nabla B\|^2_{L^2(\Omega)}\geq \|\operatorname{Sym}\nabla B\|^2_{L^2(\Omega)}$ from which we deduce $C^{T,*}_K(n)\geq 1$ and $C^{N,*}_K(n)\geq 1$ for all $n$.
		\item To obtain an improved lower bound one can just like in the case of the Gaffney inequality define the space of Killing Neumann fields $\mathcal{K}_N(\Omega)$ and the space of Killing Dirichlet fields $\mathcal{K}_D(\Omega)$ as follows
		\begin{gather}
			\label{S3Extra4}
			\mathcal{K}_N(\Omega):=\left\{Y\in \left(H^1(\Omega)\right)^n\mid \operatorname{Sym}(\nabla Y)=0\text{, }Y\parallel \partial\Omega\right\},
			\\
			\label{S3Extra5}
			\mathcal{K}_D(\Omega):=\left\{Y\in \left(H^1(\Omega)\right)^n\mid \operatorname{Sym}(\nabla Y)=0\text{, }Y\perp \partial\Omega\right\}.
		\end{gather}
		We observe that the condition $\operatorname{Sym}(\nabla Y)=0$ implies that $Y$ satisfies the Killing equations and thus $\mathcal{K}_N(\Omega)$ and $\mathcal{K}_D(\Omega)$ consist of the Killing fields on $\Omega$ which are tangent and normal to the boundary respectively. One easily checks that if $Y$ satisfies the Killing equations, i.e. $\operatorname{Sym}(\nabla Y)=0$, then $Y(x)=A\cdot x+b$ for a suitable $b\in \mathbb{R}^n$ and $A\in \operatorname{Mat}_{n\times n}(\mathbb{R})$ with $A^{\operatorname{Tr}}+A=0$ so that $\mathcal{K}_N(\Omega)$ and $\mathcal{K}_D(\Omega)$ form finite dimensional vector spaces. We observe further that if $Y\in \mathcal{K}_D(\Omega)$ and $x\in \partial\Omega$ is such that $Y(x)\neq 0$, then $s(v)\cdot v=0$ for all $v\in T_x\partial\Omega$ where $s$ denotes the shape operator of $\partial\Omega$. Equivalently $\kappa_i(x)=0$ for all $i=1,\dots,n-1$. If $Y\in \mathcal{K}_D(\Omega)\setminus \{0\}$, then the zero set of $Y$ is at most $(n-2)$-dimensional, cf. \cite[Theorem 8.1.5]{Pe16}, so that we find $\kappa_i(x)=0$ for all $1\leq i\leq n-1$ for almost every $x\in \partial\Omega$. This implies that all geodesics on $\partial\Omega$ are also geodesics in $\mathbb{R}^n$ and thus straight lines, contradicting the boundedness of $\partial\Omega$. Consequently $\mathcal{K}_D(\Omega)=\{0\}$ and no information about $C^{N,*}_K(n)$ can be derived from this approach. On the other hand there are domains with $\mathcal{K}_N(\Omega)\neq \{0\}$ so that we obtain the non-trivial bound
		\begin{gather}
			\label{S3Extra6}
			1+\sup_{\Omega\in \operatorname{Sub}_c(\mathbb{R}^n)}\sup_{Y\in \mathcal{K}_N(\Omega)\setminus\{0\}}\frac{\rho^2(\partial\Omega)\|\nabla Y\|^2_{L^2(\Omega)}}{\|Y\|^2_{L^2(\Omega)}}\leq C^{T,*}_K(n).
		\end{gather}
		We can exploit this connection to show that the bound $C^{T,*}_K(n)\leq C_1(n)$ obtained in \Cref{S3T4} is asymptotically sharp. To see this we can consider any ball $\Omega=B_r$ of any fixed radius and centred at zero. We define $Y(x_1,\dots,x_n):=(-x_2,x_1,0,\dots,0)=-x_2e_1+x_1e_2$. Clearly $Y\parallel \partial B_r$ and one easily checks that $\operatorname{Sym}(\nabla Y)=0$ so that $Y\in \mathcal{K}_N(B_r)$ for every $r>0$. Since the reach of a ball is given by its radius, we find $\rho(\partial B_r)=r$ and so we find $\rho^2(\partial\Omega)\|\nabla Y\|^2_{L^2(\Omega)}=2r^2|B_r|=2r^{n+2}|B_1|$. Further, a direct computation, working in polar coordinates, yields $\|Y\|^2_{L^2(B_r)}=\frac{2r^{n+2}}{n+2}\frac{|\partial B_1|}{n}=\frac{2r^{n+2}}{n+2}|B_1|$. Combining these identities we find $\frac{\rho^2(\partial B_r)\|\nabla Y\|^2_{L^2(B_r)}}{\|Y\|^2_{L^2(B_r)}}=n+2$ and so (\ref{S3Extra6}) yields
		\begin{gather}
			\label{S3ExtraExtra}
			C^{T,*}_K(n)\geq n+3\text{ for all }n\geq 2.
		\end{gather}
		We finally observe that $\frac{C_1(n)}{n+3}\rightarrow 1$ as $n\rightarrow \infty$ which shows that our upper bound is asymptotically optimal.
	\end{enumerate}
\end{rem}
As a consequence of (\ref{S3ExtraExtra}) and \Cref{S3T4} we obtain the following, keeping in mind that upon expanding the square we have $C_1(n)=n+3+2\sqrt{1+n}$,
\begin{cor}[High dimensional asymptotics of the tangential homogeneous Korn constant]
	\label{S3T6}
	We have $n+3\leq C^{T,*}_K(n)\leq n+3+2\sqrt{1+n}$. In particular, $\frac{C^{T,*}_K(n)}{n}\rightarrow 1$ as $n\rightarrow \infty$.
\end{cor}
\textbf{Open Question 3:} What is the exact value of the optimal constant $C^{T,*}_K(n)$ in the homogeneous Korn inequality (\ref{S1E9}) under tangent boundary conditions?
\newline
\newline
\textbf{Open Question 4:} What is the exact value of the optimal constant $C^{N,*}_K(n)$ in the homogeneous Korn inequality (\ref{S1E9}) under normal boundary conditions?
\section{Proofs of the main results}
\subsection{Quantitative Trace inequality}
In this subsection we prove a quantitative version of (\ref{S1E2}) which is the main ingredient in establishing \Cref{S3T1} and \Cref{S3T4}.
\begin{lem}[Quantitative Ehrling-trace inequality]
	\label{S4L1}
	Let $\Omega\subset\mathbb{R}^n$ be a bounded $C^{1,1}$-domain and let $\rho\equiv \rho(\partial\Omega)$ denote the reach of $\partial\Omega$. Then for every $B\in \left(H^1(\Omega)\right)^n$ and every $\epsilon>0$ the following inequality holds:
	\begin{gather}
		\label{S4E1}
		\|B\|^2_{L^2(\partial\Omega)}\leq \epsilon\rho(\partial\Omega)\|\nabla B\|^2_{L^2(\Omega)}+\frac{n+\frac{1}{\epsilon}}{\rho(\partial\Omega)}\|B\|^2_{L^2(\Omega)}.
	\end{gather}
\end{lem}
We observe that by selecting $B=(f,0,\dots,0)$ for a fixed scalar function $f\in H^1(\Omega)$, the statement of \Cref{S4L1} remains valid if we replace $B$ by a scalar function. We further note that (\ref{S4E1}) is homogeneous, i.e. both sides show the same scaling behaviour under the transformation $\Omega\mapsto \lambda \Omega$.
\begin{proof}[Proof of \Cref{S4L1}]
	We start by proving the scalar case. So let $f\in H^1(\Omega)$ be any fixed function. We observe that if $X\in C^{0,1}(\overline{\Omega},\mathbb{R}^n)$ is any $C^{0,1}$-extension of the outward unit normal $\mathcal{N}$ to $\overline{\Omega}$ we have the identity
	\begin{gather}
		\label{S4E2}
		\|f\|^2_{L^2(\partial\Omega)}=\int_{\partial\Omega}f^2d\sigma=\int_{\partial\Omega}f^2\mathcal{N}^2d\sigma=\int_{\partial\Omega}f^2X\cdot \mathcal{N}d\sigma=\int_{\Omega}2f\nabla f\cdot Xd^nx+\int_{\Omega}f^2\operatorname{div}(X)d^nx
	\end{gather}
	by means of Gauss' theorem. The goal now is to construct an appropriate extension $X$ of $\mathcal{N}$ which captures the geometry of $\partial\Omega$. To this end we consider a normal collar neighbourhood $U$ of $\partial\Omega$ of width $\rho(\partial\Omega)$. To be more precise it follows from \Cref{S2L2} that the map $\Psi:(-\rho,\rho)\times \partial\Omega\text{, }(t,x)\mapsto x-t\mathcal{N}(x)$ defines a bi-Lipschitz homeomorphism onto some open neighbourhood $V$ of $\partial\Omega$ in $\mathbb{R}^n$ and we then let $U:=V\cap \overline{\Omega}=\Psi([0,\rho))$ be our normal collar neighbourhood. We then fix any $1$-dimensional $C^{0,1}$-bump function $\psi$ with the following properties
	\begin{gather}
		\label{S4E3}
		0\leq \psi \leq 1\text{, }\operatorname{supp}(\psi)\subset (-\rho,\rho)\text{, }\psi(0)=1.
	\end{gather}
	We then define
	\begin{gather}
		\label{S4E4}
		X(y):=\begin{cases}
			\psi(t)\mathcal{N}(x) & \text{ if }y=\Psi(t,x)\in U \\
			0 & \text{ if }y\in \Omega\setminus U
		\end{cases}
	\end{gather}
	which is of class $C^{0,1}(\overline{\Omega},\mathbb{R}^n)$. By choice of $\psi$ we find $|X|\leq 1$ everywhere and so we obtain from (\ref{S4E2})
	\begin{gather}
		\label{S4E5}
		\|f\|^2_{L^2(\partial\Omega)}\leq 2\|f\|_{L^2(\Omega)}\|\nabla f\|_{L^2(\Omega)}+\|f\|^2_{L^2(\Omega)}\|\operatorname{div}(X)\|_{L^{\infty}(\Omega)}.
	\end{gather}
	We are therefore left with estimating $|\operatorname{div}(X)|$ and may focus on points $y\in U$. Now, if $\mu:\mathbb{R}^{n-1}\supset W\rightarrow \partial\Omega$, $\tilde{x}:=(x_1,\dots,x_{n-1})\mapsto \mu(\tilde{x})$ is any fixed chart, then this induces a chart on $U$ by composition with $\Psi$. Since $\Psi(t,x)=x-t\mathcal{N}(x)$ we obtain from this (by identifying the tangent basis vectors with elements of $\mathbb{R}^n$)
	\begin{gather}
		\nonumber
		\partial_t(t,x)=\mathcal{N}(x)\text{, }\partial_{x_j}(t,x)=\partial_{x_j}(x)-t\partial_{x_j}(\mathcal{N}(\mu(\tilde{x})))
	\end{gather}
	where $\partial_{x_j}(x)$ denotes the basis vector of $T_x\partial\Omega$ at $x\in \partial\Omega$ induced by the chart $\mu$ and $\partial_t(t,x),\partial_{x_j}(t,x)$ denote the basis vectors at the point $\Psi(t,x)\in U$ induced by the chart obtained from $\Psi$ by composition with $\mu$. We observe that $\mathcal{N}^2(x)=1$ along $\partial\Omega$, so that by taking the derivative we find $0=\mathcal{N}(\mu(\tilde{x}))\cdot \partial_{x_j}(\mathcal{N}(\mu(\tilde{x})))$. This implies $\partial_t(t,x)\cdot \partial_{x_j}(t,x)=0$ for all $(t,x)$ and all $j$. We notice that $X(y)=\psi(t)\partial_t$ so that in this coordinate system we get
	\begin{gather}
		\label{S4E6}
		\operatorname{div}(X)=\frac{\partial_t(\psi(t)\sqrt{\det(g)})}{\sqrt{\det{g}}}=\psi^{\prime}(t)+\psi(t)\frac{\partial_t\det(g)}{2\det(g)}.
	\end{gather}
	For a fixed $x\in \partial\Omega$ we can then let $\mu$ be a coordinate chart around $x$ such that the induced tangent vectors $\partial_{x_j}(x)$ form an orthonormal basis of the shape operator $s(x)$ of $\partial\Omega$ at $x$. We then compute
	\begin{gather}
		\nonumber
		\partial_{x_j}(t,x)\cdot \partial_{x_k}(t,x)=\delta_{jk}-t\partial_{x_k}(x)\cdot\partial_{x_j}(\mathcal{N}(\mu(\tilde{x})))-t\partial_{x_j}(x)\cdot\partial_{x_k}(\mathcal{N}(\mu(\tilde{x})))
		\\
		\nonumber
		+t^2(\partial_{x_j}(\mathcal{N}(\mu(\tilde{x})))\cdot \partial_{x_k}(\mathcal{N}(\mu(\tilde{x})))).
	\end{gather}
	Since $0=\partial_{x_j}(x)\cdot \mathcal{N}(x)$ for all $x\in \partial\Omega$, we deduce $0=(\partial_{x_k}\partial_{x_j}\mu(\tilde{x}))\cdot \mathcal{N}(x)+\partial_{x_j}(x)\cdot \partial_{x_k}(\mathcal{N}(\mu(\tilde{x})))$ and thus we get
	\begin{gather}
		\nonumber
		\partial_{x_j}(t,x)\cdot \partial_{x_k}(t,x)=\delta_{jk}+2t\mathcal{N}(x)\cdot (\partial_{x_j}\partial_{x_k}\mu)(\tilde{x})+t^2(\partial_{x_j}(\mathcal{N}(\mu(\tilde{x})))\cdot \partial_{x_k}(\mathcal{N}(\mu(\tilde{x})))).
	\end{gather}
	By choice of our basis we find $\mathcal{N}(x)\cdot (\partial_{x_j}\partial_{x_k}\mu)(\tilde{x})=h(\partial_{x_j}(x),\partial_{x_k}(x))=\kappa_j(x)\delta_{jk}$ where $h$ denotes the scalar fundamental form and $\kappa_j$ denote the principal eigenvalues with respect to $\mathcal{N}$ and where no summation convention is used in the last identity. This yields
	\begin{gather}
		\label{S4E7}
		\partial_{x_j}(t,x)\cdot \partial_{x_k}(t,x)=\delta_{jk}+2t\kappa_j\delta_{jk}+t^2(\partial_{x_j}(\mathcal{N}(\mu(\tilde{x})))\cdot \partial_{x_k}(\mathcal{N}(\mu(\tilde{x}))))
	\end{gather}
	where no summation convention is used. We are left with expressing the remaining term in terms of geometric quantities related to $\partial\Omega$. This can be done by exploiting the fact that $\mathcal{N}\cdot (\partial_{x_j}\mathcal{N})=0$, since $|\mathcal{N}|^2=1$ along $\partial\Omega$, and expanding $\partial_{x_j}\mathcal{N}$ in terms of our orthonormal basis $\{\mathcal{N}(x),\partial_{x_1}(x),\dots,\partial_{x_{n-1}}(x)\}$
	\begin{gather}
		\nonumber
		\partial_{x_j}(\mathcal{N}(\mu(\tilde{x})))=\sum_{i=1}^{n-1}\left((\partial_{x_i}(x))\cdot \partial_{x_j}(\mathcal{N}(\tilde{x}))\right)\partial_{x_i}(x).
	\end{gather}
	We have already seen that $\partial_{x_i}(x)\cdot\partial_{x_j}(\mathcal{N}(\mu(\tilde{x})))=-\mathcal{N}(x)\cdot (\partial_{x_i}\partial_{x_j}\mu)(\tilde{x})=-\kappa_j\delta_{ji}$. Inserting this above we deduce
	\begin{gather}
		\nonumber
		\partial_{x_j}(\mathcal{N}(\mu(\tilde{x})))=-\kappa_j\partial_{x_j}(x)
	\end{gather}
	and consequently (\ref{S4E7}) becomes
	\begin{gather}
		\label{S4E8}
		\partial_{x_j}(t,x)\cdot \partial_{x_k}(t,x)=\delta_{jk}(1+2t\kappa_j+t^2\kappa^2_j)=\delta_{jk}(1+t\kappa_j)^2.
	\end{gather}
	Consequently we find $\operatorname{det}(g)(t,x)=\Pi_{j=1}^{n-1}(1+t\kappa_j(x))^2$. By choice of our coordinate system this formula is valid at the fixed point $x\in \partial\Omega$ and for all $t$. We hence may insert this in (\ref{S4E6}) and deduce
	\begin{gather}
		\nonumber
		\operatorname{div}(X)(y)=\psi^{\prime}(t)+\psi(t)\sum_{j=1}^{n-1}\frac{\kappa_j(x)}{1+t\kappa_j(x)}.
	\end{gather}
	Since the principal curvatures may be bounded above by $\rho^{-1}(\partial\Omega)$, cf. \cite[Remark 2.8]{Dal18}, and we consider $0\leq t<\rho$, we find $\frac{1}{1+t\kappa_j}\leq \frac{1}{1-\frac{t}{\rho}}$ and find
	\begin{gather}
		\label{S4E9}
		|\operatorname{div}(X)(y)|\leq |\psi^{\prime}(t)|+\frac{|\psi(t)|}{\rho(\partial\Omega)}\frac{n-1}{1-\frac{t}{\rho(\partial\Omega)}}.
	\end{gather}
To obtain explicit constants in (\ref{S4E9}) we make a specific choice for $\psi$. Fix $0<\alpha<1$ and set
\begin{gather}
	\label{S4E10}
	\psi(t):=\begin{cases}
		1-\frac{t}{\alpha \rho}& \text{, }0\leq t\leq \alpha\rho \\
		0 & \text{ for }t>\alpha\rho
	\end{cases}
\end{gather}
One readily checks that $\psi\in C^{0,1}([0,\infty))$, $0\leq \psi\leq 1$, $\psi(0)=1$ and $\psi(t)=0$ for $|t|\geq \alpha \rho$. One can therefore find a suitable $C^{0,1}$-extension of $\psi$ for $t<0$ to ensure it satisfies (\ref{S4E3}). With this specific choice of $\psi(t)$ we obtain for $0\leq t\leq \alpha \rho$, $\frac{\psi(t)}{1-\frac{t}{\rho(\partial\Omega)}}= \frac{1-\frac{t}{\alpha \rho}}{1-\frac{t}{\rho}}$. We can set $s:=\frac{t}{\rho}$ and $\lambda(s):=\frac{1-\frac{s}{\alpha}}{1-s}$ and compute $\lambda^\prime(s)=-\frac{1-\alpha}{\alpha(1-s)^2}\leq 0$ for all $0<\alpha<1$ so that $\lambda$ is decreasing with $\lambda(0)=1$ and $\lambda(\alpha)=0$ from which conclude that $0\leq \frac{\psi(t)}{1-\frac{t}{\rho(\partial\Omega)}}=\lambda(s)\leq 1$ for all $0<\alpha<1$. In addition, $\|\psi^\prime\|_{L^{\infty}}\leq \frac{1}{\alpha\rho}$ so that we obtain the estimate
\begin{gather}
	\nonumber
	|\operatorname{div}(X)(y)|\leq \frac{1}{\rho(\partial\Omega)}\left(\frac{1}{\alpha}+n-1\right)
\end{gather}
where $0<\alpha<1$ can be arbitrarily chosen. Inserting this in (\ref{S4E5}) we find
\begin{gather}
	\label{S4E13}
	\|f\|^2_{L^2(\partial\Omega)}\leq 2\|f\|_{L^2(\Omega)}\|\nabla f\|_{L^2(\Omega)}+\frac{n-1+\frac{1}{\alpha}}{\rho(\partial\Omega)}\|f\|^2_{L^2(\Omega)}.
\end{gather}
Since this construction works for every $0<\alpha<1$ we can now take the limit $\alpha\nearrow 1$ and find
\begin{gather}
	\label{S4EExtra1}
	\|f\|^2_{L^2(\partial\Omega)}\leq 2\|f\|_{L^2(\Omega)}\|\nabla f\|_{L^2(\Omega)}+\frac{n}{\rho(\partial\Omega)}\|f\|^2_{L^2(\Omega)}.
\end{gather}
Lastly we find $2\|f\|_{L^2(\Omega)}\|\nabla f\|_{L^2(\Omega)}=2\frac{\|f\|_{L^2(\Omega)}}{\sqrt{\rho}}\left(\sqrt{\rho}\|\nabla f\|_{L^2(\Omega)}\right)\leq \epsilon\rho\|\nabla f\|^2_{L^2(\Omega)}+\frac{\|f\|^2_{L^2(\Omega)}}{\epsilon\rho}$ which in combination with (\ref{S4EExtra1}) proves the lemma in case $B=(f,0,\dots,0)$. The general case follows by applying the inequality to each component and taking the sum over all these inequalities.
\end{proof}
\subsection{Proof of the quantitative Gaffney inequality}
We start with the following lemma, which can be found in a more abstract setting in \cite[Theorem 2.1.5]{S95}, see also \cite[Lemma 2.11]{ABDG98} for the $C^{1,1}$-setting for $3$-dimensional domains as well as \cite[Theorem 3.1.1.1 \& Theorem 3.1.1.2]{Gris85} for related results. The calculations simplify in the Euclidean vector field setting, so that we include the simplified proof for convenience of the reader. 
\begin{lem}
	\label{S4L2}
	Let $\Omega\subset\mathbb{R}^n$ be a bounded domain with $C^{1,1}$-boundary. Let $s$ and $H$ denote the shape operator and the mean curvature of $\partial\Omega$ with respect to the outward unit normal respectively. Then the following holds for every $B\in \left(H^1(\Omega)\right)^n$
	\begin{gather}
		\|\nabla B\|^2_{L^2(\Omega)}=\|\operatorname{curl}(B)\|^2_{L^2(\Omega)}+\|\operatorname{div}(B)\|^2_{L^2(\Omega)}+\int_{\partial\Omega}s(B)\cdot Bd\sigma\text{, if }B\parallel \partial\Omega,
		\\
		\|\nabla B\|^2_{L^2(\Omega)}=\|\operatorname{curl}(B)\|^2_{L^2(\Omega)}+\|\operatorname{div}(B)\|^2_{L^2(\Omega)}+(n-1)\int_{\partial\Omega}H|B|^2d\sigma\text{, if }B\perp \partial\Omega.
	\end{gather}
\end{lem}
\begin{proof}[Proof of \Cref{S4L2}]
	We recall that $(\operatorname{curl}(B))_{ij}=\frac{\partial_iB^j-\partial_jB^i}{\sqrt{2}}$ so that we can compute for any $B\in \left(H^2(\Omega)\right)^n$
	\begin{gather}
		\nonumber
		\|\operatorname{curl}(B)\|^2_{L^2(\Omega)}=\sum_{i,j=1}^n\int_{\Omega}\frac{(\partial_iB^j-\partial_jB^i)^2}{2}dx^n=\sum_{i,j=1}^n\int_{\Omega}\frac{(\partial_iB^j)^2+(\partial_jB^i)^2-2(\partial_jB^i)(\partial_iB^j)}{2}dx^n
		\\
		\nonumber
		=\|\nabla B\|^2_{L^2(\Omega)}-\sum_{i,j=1}^n\int_{\Omega}(\partial_jB^i)(\partial_iB^j)dx^n=\|\nabla B\|^2_{L^2(\Omega)}+\int_{\Omega}B\cdot \nabla \operatorname{div}(B)dx^n-\int_{\partial\Omega}\mathcal{N}\cdot \nabla_BBd\sigma
		\\
		\label{S4E14}
		=\|\nabla B\|^2_{L^2(\Omega)}-\|\operatorname{div}(B)\|^2_{L^2(\Omega)}+\int_{\partial\Omega}\operatorname{div}(B)(\mathcal{N}\cdot B)-\mathcal{N}\cdot \nabla_BBd\sigma
	\end{gather}
	where $\nabla_BB=\sum_{i,j=1}^nB^i(\partial_iB^j)e_j$ is the standard covariant derivative of $B$ with respect to $B$.
	
	If $B\parallel \partial\Omega$, then $\mathcal{N}\cdot B=0$ and $\mathcal{N}\cdot \nabla_BB=s(B)\cdot B$ where $s$ denotes the shape operator of $\partial\Omega$ w.r.t. $\mathcal{N}$. Hence in this case (\ref{S4E14}) becomes
	\begin{gather}
		\label{S4E15}
		\|\nabla B\|^2_{L^2(\Omega)}=\|\operatorname{curl}(B)\|^2_{L^2(\Omega)}+\|\operatorname{div}(B)\|^2_{L^2(\Omega)}+\int_{\partial\Omega}s(B)\cdot Bd\sigma.
	\end{gather}
	If on the other hand $B\perp \partial\Omega$, then we can let $\widetilde{\mathcal{N}}$ be any $C^{0,1}$-extension of $\mathcal{N}$ to $\Omega$ such that $|\widetilde{\mathcal{N}}|=1$ in some open neighbourhood of $\partial\Omega$ and we can expand $B=(B\cdot \widetilde{\mathcal{N}})\widetilde{\mathcal{N}}+B^\perp$ where $B^\perp\cdot \widetilde{\mathcal{N}}=0$ in this neighbourhood. A direct computation yields
	\begin{gather}
		\nonumber
		(\mathcal{N}\cdot B)\operatorname{div}(B)-\mathcal{N}\cdot \nabla_BB=(\mathcal{N}\cdot B)^2\operatorname{div}(\widetilde{\mathcal{N}})-(\mathcal{N}\cdot B)^2\widetilde{\mathcal{N}}\cdot \nabla_{\widetilde{\mathcal{N}}}\widetilde{\mathcal{N}}+(\mathcal{N}\cdot B)\left(\operatorname{div}(B^\perp)-\mathcal{N}\cdot \nabla_{\mathcal{N}}B^\perp\right).
	\end{gather}
	We notice that $\operatorname{div}(B^\perp)-\mathcal{N}\cdot \nabla_{\mathcal{N}}B^\perp=\operatorname{div}_{\partial\Omega}(B^{\perp}|_{\partial\Omega})=0$ since $B^\perp|_{\partial\Omega}=0$ and where $\operatorname{div}_{\partial\Omega}$ denotes the intrinsic divergence on $\partial\Omega$. In addition we have the identity $\widetilde{\mathcal{N}}\cdot \nabla_{\widetilde{\mathcal{N}}}\widetilde{\mathcal{N}}=\widetilde{\mathcal{N}}\cdot \nabla \frac{|\widetilde{\mathcal{N}}|^2}{2}=0$ where the last identity follows from the fact that $|\widetilde{\mathcal{N}}|=1$ in a neighbourhood of $\partial\Omega$. We conclude $(\mathcal{N}\cdot B)\operatorname{div}(B)-\mathcal{N}\cdot \nabla_BB=(\mathcal{N}\cdot B)^2\operatorname{div}(\widetilde{\mathcal{N}})=|B|^2\operatorname{div}(\widetilde{\mathcal{N}})$ where we used that $B\perp\partial\Omega$ and where $\widetilde{\mathcal{N}}$ is any $C^{0,1}$-extension of $\mathcal{N}$ which is of unit length in a neighbourhood of $\partial\Omega$. It is then standard, cf. \cite[Problem 8.2 (b)]{L18}, that $\operatorname{div}(\widetilde{\mathcal{N}})=-(n-1)H$ where $H$ is the mean curvature with respect to $\mathcal{N}$. This, in combination with (\ref{S4E14}), implies
	\begin{gather}
		\nonumber
		\|\nabla B\|^2_{L^2(\Omega)}=\|\operatorname{curl}(B)\|^2_{L^2(\Omega)}+\|\operatorname{div}(B)\|^2_{L^2(\Omega)}+(n-1)\int_{\partial\Omega}|B|^2 Hd\sigma.
	\end{gather}
	For $B\in \left(H^1(\Omega)\right)^n$, the statement then follows by approximation.
\end{proof}
\begin{proof}[Proof of \Cref{S3T1}]
	We start from \Cref{S4L2} and observe that we have the following pointwise estimates
	\begin{gather}
		\label{S4E18}
		|s(B)\cdot B|\leq \max_{1\leq i\leq n-1}\sup_{x\in \partial\Omega}|\kappa_i(x)||B|^2\leq \frac{|B|^2}{\rho(\partial\Omega)}\text{ if }B\parallel \partial\Omega,
		\\
		\label{S4E19}
		|(n-1)H|\leq (n-1)\max_{1\leq i\leq n-1}\sup_{x\in \partial\Omega}|\kappa_i(x)|\leq \frac{n-1}{\rho(\partial\Omega)}
	\end{gather}
	where we used that $\rho^{-1}(\partial\Omega)$ provides an upper bound on the curvatures of $\partial\Omega$. Now let $B\in \left(H^1(\Omega)\right)^n$ with either $B\parallel \partial\Omega$ or $B\perp \partial\Omega$ and set $c_n:=\begin{cases}
		1 & \text{ if }B\parallel \partial\Omega, \\
		n-1 & \text{ if }B\perp \partial\Omega,
	\end{cases}$. Then \Cref{S4L2} implies
	\begin{gather}
		\label{S4E20}
		\|\nabla B\|^2_{L^2(\Omega)}\leq \|\operatorname{curl}(B)\|^2_{L^2(\Omega)}+\|\operatorname{div}(B)\|^2_{L^2(\Omega)}+\frac{c_n}{\rho(\partial\Omega)}\|B\|^2_{L^2(\partial\Omega)}.
			\end{gather}
			We can then combine (\ref{S4E20}) with \Cref{S4L1} to deduce
			\begin{gather}
				\nonumber
				\|\nabla B\|^2_{L^2(\Omega)}\leq \|\operatorname{curl}(B)\|^2_{L^2(\Omega)}+\|\operatorname{div}(B)\|^2_{L^2(\Omega)}+c_n\epsilon \|\nabla B\|^2_{L^2(\Omega)}+c_n\frac{n+\frac{1}{\epsilon}}{\rho^2(\partial\Omega)}\|B\|^2_{L^2(\Omega)}\text{ for all }\epsilon>0.
			\end{gather}
			We rearrange terms and find
			\begin{gather}
				\label{S4E21}
				\|\nabla B\|^2_{L^2(\Omega)}\leq \frac{\|\operatorname{curl}(B)\|^2_{L^2(\Omega)}+\|\operatorname{div}(B)\|^2_{L^2(\Omega)}}{1-\epsilon c_n}+\frac{c_n}{1-c_n \epsilon}\frac{n+\frac{1}{\epsilon}}{\rho^2(\partial\Omega)}\|B\|^2_{L^2(\Omega)}\text{ for all }0<\epsilon<\frac{1}{c_n}.
			\end{gather}
			We now pick $\epsilon=\frac{1}{c_n+\sqrt{c_n}\sqrt{c_n+n}}$ and observe that $\epsilon<\frac{1}{c_n}$ as required. We verify by direct computation $\frac{c_n}{1-c_n\epsilon}\left(n+\frac{1}{\epsilon}\right)=(c_n+\sqrt{c_n}\sqrt{c_n+n})^2$ and $\frac{1}{1-\epsilon c_n}\leq \frac{c_n}{1-\epsilon c_n}\left(n+\frac{1}{\epsilon}\right)$ for our choice of $\epsilon$. From this we deduce
				\begin{gather}
					\label{S4E22}
					\|\nabla B\|^2_{L^2(\Omega)}\leq \left(c_n+\sqrt{c_n}\sqrt{c_n+n}\right)^2\left(\|\operatorname{curl}(B)\|^2_{L^2(\Omega)}+\|\operatorname{div}(B)\|^2_{L^2(\Omega)}+\frac{\|B\|^2_{L^2(\Omega)}}{\rho^2(\partial\Omega)}\right).
				\end{gather}
				The theorem now follows immediately from (\ref{S4E22}).
\end{proof}

\subsection{Proof of the quantitative Korn inequality}
The following is the corresponding analogue of \Cref{S4L2} for the situation of the symmetric gradient
\begin{lem}
	\label{S4L3}
	Let $\Omega\subset\mathbb{R}^n$ be a bounded domain with $C^{1,1}$-boundary. Let $s$ and $H$ denote the shape operator and mean curvature of $\partial\Omega$ with respect to the outward unit normal respectively. Then the following holds for every $B\in \left(H^1(\Omega)\right)^n$
	\begin{gather}
		\label{S4E23}
		\|\nabla B\|^2_{L^2(\Omega)}=2\|\operatorname{Sym}(\nabla B)\|^2_{L^2(\Omega)}-\|\operatorname{div}(B)\|^2_{L^2(\Omega)}-\int_{\partial\Omega}s(B)\cdot Bd\sigma\text{, if }B\parallel \partial\Omega,
		\\
		\label{S4E24}
		\|\nabla B\|^2_{L^2(\Omega)}=2\|\operatorname{Sym}(\nabla B)\|^2_{L^2(\Omega)}-\|\operatorname{div}(B)\|^2_{L^2(\Omega)}-(n-1)\int_{\partial\Omega}H|B|^2d\sigma\text{, if }B\perp\partial\Omega.
	\end{gather}
\end{lem}
\begin{proof}[Proof of \Cref{S4L3}]
	The proof of \Cref{S4L2} applies mutatis mutandis by considering the expansion $2\|\operatorname{Sym}(\nabla B)\|^2_{L^2(\Omega)}=\|\nabla B\|^2_{L^2(\Omega)}+\sum_{i,j=1}^n\int_{\Omega}(\partial_jB^i)(\partial_iB^j)dx^n$. The only difference now is that the mixed term is of a different sign.
\end{proof}
\begin{proof}[Proof of \Cref{S3T4}]
	Fix $B\in \left(H^1(\Omega)\right)^n$ with either $B\parallel \partial\Omega$ or $B\perp \partial\Omega$. Since $\|\operatorname{div}(B)\|^2_{L^2(\Omega)}\geq 0$ we deduce from (\ref{S4E18}),(\ref{S4E19}) and \Cref{S4L3}
	\begin{gather}
		\nonumber
		\|\nabla B\|^2_{L^2(\Omega)}\leq 2\|\operatorname{Sym}(\nabla B)\|^2_{L^2(\Omega)}+\frac{c_n}{\rho(\partial\Omega)}\|B\|^2_{L^2(\partial\Omega)}
	\end{gather}
	where $c_n=1$ if $B\parallel \partial\Omega$ and $c_n=n-1$ if $B\perp\partial\Omega$. Making use of \Cref{S4L1} yields
	\begin{gather}
		\nonumber
		\|\nabla B\|^2_{L^2(\Omega)}\leq 2\|\operatorname{Sym}(\nabla B)\|^2_{L^2(\Omega)}+c_n\epsilon\|\nabla B\|^2_{L^2(\Omega)}+\frac{c_n}{\rho^2(\partial\Omega)}\left(n+\frac{1}{\epsilon}\right)\|B\|^2_{L^2(\Omega)}\text{ for all }\epsilon>0.
	\end{gather}
	Rearranging terms yields
	\begin{gather}
		\label{S4E25}
		\|\nabla B\|^2_{L^2(\Omega)}\leq \frac{2\|\operatorname{Sym}(\nabla B)\|^2_{L^2(\Omega)}}{1-c_n\epsilon}+\frac{c_n}{1-c_n\epsilon}\frac{n+\frac{1}{\epsilon}}{\rho^2(\partial\Omega)}\|B\|^2_{L^2(\Omega)}\text{ for all }0<\epsilon<\frac{1}{c_n}.
	\end{gather}
	We recognise that (\ref{S4E25}) has the same structure as (\ref{S4E21}) except for the additional factor $2$ in the expression $\frac{2}{1-\epsilon c_n}$. We can now make the same choice for $\epsilon$, i.e. consider $\epsilon=\frac{1}{c_n+\sqrt{c_n}\sqrt{c_n+n}}$ and observe that $\frac{2}{1-c_n\epsilon}\leq \frac{c_n}{1-c_n\epsilon}\left(n+\frac{1}{\epsilon}\right)=\left(c_n+\sqrt{c_n}\sqrt{c_n+n}\right)^2$ so that we arrive at
	\begin{gather}
		\label{S4E26}
		\|\nabla B\|^2_{L^2(\Omega)}\leq \left(c_n+\sqrt{c_n}\sqrt{c_n+n}\right)^2\left(\|\operatorname{Sym}(\nabla B)\|^2_{L^2(\Omega)}+\frac{\|B\|^2_{L^2(\Omega)}}{\rho^2(\partial\Omega)}\right)
	\end{gather}
	which proves the theorem.
\end{proof}
\section*{Acknowledgements}
The research was supported in part by the MIUR Excellence Department Project awarded to Dipartimento di Matematica, Università di Genova, CUP D33C23001110001. This work has been partly supported by the Inria AEX StellaCage. 
\bibliographystyle{plain}
\bibliography{mybibfileNOHYPERLINK}
\footnotesize
\end{document}